\documentclass{amsart}

\usepackage{a4wide}
\usepackage{amscd}
\usepackage{amssymb}
\usepackage{amsthm}
\usepackage{amsmath}
\usepackage{color}
\usepackage{times}
\usepackage{latexsym}
\usepackage{amsfonts}
\usepackage{graphicx}
\usepackage{graphicx, subfigure}
\usepackage{hyperref} 

\newtheorem{theorem}{Theorem}
\theoremstyle{plain}

\newtheorem{definition}[theorem]{Definition}
\newtheorem{example}[theorem]{Example}

\newtheorem{lemma}[theorem]{Lemma}

\newtheorem{proposition}[theorem]{Proposition}
\newtheorem{remark}[theorem]{Remark}

\numberwithin{equation}{section}
\numberwithin{theorem}{section}

\newcommand{\R}{\ensuremath{\mathbb{R}}}
\newcommand{\E}{\ensuremath{\mathbb{E}}}
\newcommand{\N}{\ensuremath{\mathbb{N}}}
\newcommand{\C}{\ensuremath{\mathbb{C}}}
\newcommand{\Sg}{\ensuremath{\mathbb{S}}}
\newcommand{\Le}{{\mathcal{L}}}
\newcommand{\Y}{{\mathcal{Y}}}

\newcommand{\X}{{\mathcal{X}}}

\newcommand{\dom}{\mathrm{dom}} 
\newcommand{\tr}{\mathrm{Tr}} 

\def\e{{\mathrm{e}}}
\allowdisplaybreaks

\title{Ornstein-Uhlenbeck processes in Hilbert space with non-Gaussian stochastic volatility}
\author[Benth]{Fred Espen Benth}
\address[Fred Espen Benth]{\\
Department of Mathematics\\
University of Oslo\\
P.O. Box 1053, Blindern\\
N--0316 Oslo, Norway \\
and \\
Centre for Advanced Study \\
Drammensveien 78 \\
N-0271 Oslo, Norway}
\email[]{fredb\@@math.uio.no}
\urladdr{http://folk.uio.no/fredb/}
\author[R\"udiger]{Barbara R\"udiger}
\address[Barbara R\"udiger]{ \\
Fachbereich C \\
Bergische Universit\"at Wuppertal \\
D-42119 Wuppertal, Germany}
\email[]{ruediger\@@uni-wuppertal.de}
\author[S\"uss]{Andre S\"uss}
\address[Andre S\"uss]{\\
Centre for Advanced Study \\
Drammensveien 78 \\
N-0271 Oslo, Norway}
\email[]{suess.andre\@@web.de}

\date{\today}
\thanks{F. E. Benth acknowledges financial support from "Managing Weather Risk in Energy Markets (MAWREM)", funded by the Norwegian 
Research Council. B. R\"udiger is grateful for the financial support from the Centre for Advanced Study (CAS) in Oslo. We are grateful for comments by Josef Teichmann.}
\begin{document}
\maketitle
\begin{abstract}
We propose a non-Gaussian operator-valued extension of the Barndorff-Nielsen and Shephard stochastic volatility dynamics, 
defined as the square-root of an operator-valued Ornstein-Uhlenbeck process with L\'evy noise and bounded drift.
We derive conditions for the positive
definiteness of the Ornstein-Uhlenbeck process, where in particular we must restrict to operator-valued L\'evy processes with
"non-decreasing paths". It turns out that the volatility model allows for an explicit calculation of its characteristic function, showing 
an affine structure. We introduce another Hilbert space-valued Ornstein-Uhlenbeck process with Wiener noise perturbed by this class of 
stochastic volatility dynamics. Under a strong commutativity 
condition between the covariance operator of the Wiener process and the stochastic volatility,
we can derive an analytical expression for the characteristic functional of the Ornstein-Uhlenbeck process perturbed by stochastic volatility
if the noises are independent. 
The case of operator-valued compound Poisson processes as driving noise in the volatility is discussed as a particular example of interest.
We apply our results to futures prices in commodity markets, where we discuss our proposed stochastic volatility model in light of
ambit fields. 
\end{abstract} 

\section{Introduction}
In this paper we introduce and analyse an Ornstein-Uhlenbeck (OU) process 
$$
dX(t)=\mathcal{A} X(t)\,dt+\sigma(t)\,dB(t)
$$ 
taking values in a separable Hilbert space $H$. Here, $\mathcal{A}$ is a densely defined unbounded operator on $H$,
$B$ is an $H$-valued Wiener process and $\sigma(t)$ is a predictable operator-valued process being integrable with respect to 
$B$. We shall be concerned with a particular class of stochastic volatility models $\sigma(t)$ of a non-Gaussian nature.
   
OU processes with values in Hilbert space provide a natural infinite dimensional formulation for many linear (parabolic) 
stochastic partial differential equations
(see, e.g., Da Prato and Zabczyk~\cite{DaPZ}, Gawarecki and Mandrekar~\cite{GM} and Peszat and Zabczyk~\cite{PZ}). 
Our main motivation for studying Hilbert space-valued OU processes comes from the modelling of futures prices in commodity markets, where the
dynamics follow a class of hyperbolic 
stochastic partial differential equations (see Benth and Kr\"uhner~\cite{BK-HJM,BK2}).

Barndorff-Nielsen and Shephard~\cite{BNS} proposed a flexible class of stochastic volatility (SV) models based on real-valued OU processes
driven by a subordinator (a pure-jump L\'evy process with non-negative drift and positive jumps). This class, which we name 
the BNS SV model, has been applied to model financial time series like exchange rates and stock prices (see e.g. Barndorff-Nielsen and
Shephard ~\cite{BNS}). Benth~\cite{B-mf} proposed the BNS SV model in an exponential mean-reversion dynamics to model gas prices collected
from the UK market. Later, Benth and Vos \cite{BV1,BV2} extended this to a multifactor framework to model prices in
energy markets. Their extension of the BNS SV model to a multivariate context is based on the work by
Barndorff-Nielsen and Stelzer~\cite{BNStelzer}. There are several papers dealing, both empirically and theoretically, with stochastic volatility 
in commodity prices (see e.g., Geman \cite{Geman}, Hikspoors and Jaimungal~\cite{HJ} and Schwartz and Trolle~\cite{ST}).

In the present paper we lift the multivariate BNS SV model by Barndorff-Nielsen and Stelzer~\cite{BNStelzer} to 
an operator-valued stochastic process, providing a very general stochastic volatility dynamics. In particular, 
we consider the "stochastic variance process" $\mathcal{Y}(t)$ taking values in the space of Hilbert-Schmidt operators on
$H$,
$$
d\mathcal{Y}(t)=\C \mathcal{Y}(t)\,dt+d\mathcal{L}(t)\,,
$$
where $\mathcal{L}$ is a square-integrable L\'evy process in the space of Hilbert-Schmidt operators on $H$ and $\C$ a 
bounded operator on the same space. We state conditions on $\C$ and $\mathcal{L}$ to ensure that $\mathcal{Y}$ is a 
non-negative definite self -adjoint operator, and in this case we define $\sigma(t):=\mathcal{Y}^{1/2}(t)$. In fact, 
the paths of the process $t\mapsto (\mathcal{L}(t)f,f)_H$ must be increasing for every $f\in H$ to have non-negative definite
$\mathcal Y$. This property is analogous to the assumption the  real-valued BNS SV model is driven by a subordinator
process. since $t\mapsto (\mathcal L(t)f,f)_H$ is equal to the scalar product of $\mathcal{L}(t)$ with $f\otimes f$ in the space
of Hilbert-Schmidt operators on $H$, and thus a real-valued L\'evy process with  non-decreasing paths (i.e., a
subordinator). We say that $\Le$ has "non-decreasing paths" and we show that 
such L\'evy processes have a continuous martingale part with covariance operator
having all symmetric Hilbert-Schmidt operators in its kernel.   

As a particular example a compound Poisson process 
is considered, where the jumps are defined to be the tensor product of a Hilbert space valued Gaussian random variable with itself.
We demonstrate that such a model leads to Gamma distributed jumps for certain interesting real-valued projections of the L\'evy process.
Furthermore, from a result of Fraisse and Viguier-Pla~\cite{FVP} the jumps will in general be Wishart distributed in infinite dimensions, 
and we can compute the characteristic functional of $\Le$ for self-adjoint test operators. 

Our operator-valued BNS SV model $\mathcal{Y}$ has a convenient affine structure, and we can compute its characteristic function.  
Moreover, if $\Le$ is independent of $B$, it is possible to derive an analytical expression for the characteristic function of the 
OU-process $X(t)$ in terms of the semigroups associated with the drift in $X$ and $\mathcal{Y}$ and the characteristic functional of 
$\Le$. To achieve this result, we must impose a rather strong commutativity condition between the covariance operator of the Wiener noise
$B$ and the stochastic volatility $\Y^{1/2}$. We find that $X$ is affine in itself and the stochastic volatility.  
Also, we show that the "mean-reversion adjusted returns" of $X$ are $H$-valued conditional Gaussian random variables, 
if these are conditioned on the volatility $\Y^{1/2}$ , which can be considered to be an  observable in a simplified filtering problem (see Remark \ref{Remarkfiltering} in Section 3). The "mean-reversion
adjusted returns" are defined as the increments of $X$ corrected by the semigroup of $\C$. 

We relate our general analysis to commodity futures markets.
In this respect, we focus on a process $X$ defined on a specific Hilbert space of functions on $\R_+$, the positive real-line, and with the unbounded operator in the drift being $\mathcal{A}=\partial/\partial x$. Then, $X(t,x)$ can be interpreted as the futures price at time $t\geq 0$
for a contract delivering the commodity at time $x\geq 0$, with a dynamics specified under the Heath-Jarrow-Morton-Musiela (HJMM) modelling
paradigm (see Heath, Jarrow and Morton \cite{HJM} and Musiela \cite{Mus}). We connect our general SV modelling approach to the analysis in
Benth and Kr\"uhner~\cite{BK-HJM,BK2} and the ambit field approach in 
Barndorff-Nielsen, Benth and Veraart~\cite{BNBV-forward,BNBV-cross}. 
We remark that this discussion can be extended to forward rate modelling under the
HJM paradigm in fixed-income theory (see Filipovic~\cite{filipovic} and Carmona and Theranchi \cite{CT} for an analysis of HJM models in 
infinite dimensions for fixed-income markets.).

Our results are presented as follows: In the next section we introduce the operator-valued BNS SV model and analyse its
properties. Section 3 defines the volatility-modulated OU process $X$ along with a discussion of its characteristics. Finally, in Section 4, we
discuss our model $X$ in the context of commodity futures price modelling.

\section{Operator-valued BNS stochastic volatility model}

Throughout the paper, $(\Omega,\mathcal{F},\{\mathcal{F}_t\}_{t\geq 0}, P)$ is a given filtered probability space.
Let $H$ be a separable Hilbert space with inner product denoted by $(\cdot,\cdot)_H$ and associated norm
$|\cdot|_H$. 
Introduce $\mathcal{H}:=L_{\text{HS}}(H)$, the space of Hilbert-Schmidt operators on $H$ into itself, with the usual inner product denoted
by $\langle\cdot,\cdot\rangle_{\mathcal{H}}$ and associated norm $\|\cdot\|_{\mathcal{H}}$. 
As $H$ is a separable Hilbert space, $\mathcal{H}$ becomes a separable Hilbert space as well.

Introduce $\C\in L(\mathcal{H})$, that is, a bounded linear operator from $\mathcal{H}$ into itself. In this paper, we shall pay
particular attention to two specific cases of $\C$, namely, 
the operator 
\begin{equation}
\label{def-C1}
\C_1:\mathcal{H}\rightarrow\mathcal{H},\qquad \mathcal{T}\mapsto\mathcal{C}\mathcal{T}\mathcal{C}^*
\end{equation} 
 or
the operator
\begin{equation}
\label{def-C2}
\C_2:\mathcal{H}\rightarrow\mathcal{H},\qquad \mathcal{T}\mapsto\mathcal{C}\mathcal{T}+\mathcal{T}\mathcal{C}^*\,.
\end{equation} 
Here, $\mathcal{C}\in L(H)$, $L(H)$ denoting the space of bounded linear operators in $H$ into itself.
We shall exclusively focus on $\mathcal{C}\neq 0$. 
The following lemma provides us with crucial properties for $\C_i, i=1,2$:
\begin{lemma}
\label{lemma:basic-prop-C}
It holds that $\C_i\in L(\mathcal{H})$ for $\C_i$ defined in \eqref{def-C1} and \eqref{def-C2}, with $\|\C_1\|_{\text{op}}\leq \|\mathcal{C}\|_{\text{op}}^2$ and
$\|\C_2\|_{\text{op}}\leq 2\|\mathcal{C}\|_{\text{op}}$. 
Moreover, $(\C_i\mathcal{T})^*=\C_i\mathcal{T}^*$ for every $\mathcal{T}\in\mathcal{H}$ and $i=1,2$. 
\end{lemma}
\begin{proof}
For $\mathcal{S},\mathcal{T}\in\mathcal{H}$, $\C_i(\mathcal S+\mathcal T)=\C_i\mathcal S+\C_i\mathcal T$, where $i=1,2$. Hence, linearity holds. Moreover,
for an orthonormal basis $\{e_n\}_{n\in\N}$ in $H$,
\begin{align*}
\|\C_1\mathcal{T}\|_{\mathcal{H}}^2&=\|\mathcal{C}\mathcal{T}\mathcal{C}^*\|^2_{\mathcal{H}} \\
&=\sum_{n=1}^{\infty}|\mathcal{C}\mathcal{T}\mathcal{C}^*e_n|_H^2 \\
&\leq \|\mathcal{C}\|^2_{\text{op}}\sum_{n=1}^{\infty}|\mathcal{T}\mathcal{C}^*e_n|^2_H \\
&\leq \|\mathcal{C}\|^4_{\text{op}}\sum_{n=1}^{\infty}|\mathcal{T}e_n|^2_H \\
&=\|\mathcal{C}\|^4_{\text{op}}\|\mathcal{T}\|_{\mathcal{H}}^2\,.
\end{align*}
Here we have used that $\|\mathcal{T}\mathcal{C}^*\|_{\mathcal{H}}=\|\mathcal{C}\mathcal{T}^*\|_{\mathcal{H}}$. 
For $\C_2$, we have by the triangle inequality,
\begin{align*}
\|\C_2\mathcal{T}\|_{\mathcal{H}}&=\|\mathcal{C}\mathcal{T}+\mathcal{T}\mathcal{C}^*\|_{\mathcal{H}} \\
&\leq \|\mathcal{C}\mathcal{T}\|_{\mathcal{H}}+\|\mathcal{T}\mathcal{C}^*\|_{\mathcal{H}} \\
&\leq\|\mathcal{C}\|_{\text{op}} \|\mathcal{T}\|_{\mathcal{H}}+\|\mathcal{C}\|_{\text{op}}\|\mathcal{T}^*\|_{\mathcal{H}} \\
&=2\|\mathcal{C}\|_{\text{op}}\|\mathcal{T}\|_{\mathcal{H}}\,.
\end{align*}
Hence, the first claim of the lemma holds. 

For  $\mathcal{T}\in\mathcal{H}$, it follows that 
\begin{align*}
(\C_1\mathcal{T}f,g)_H&=(\mathcal{C}\mathcal{T}\mathcal{C}^*f,g)_H \\
&=(f,\mathcal{C}\mathcal{T}^*\mathcal{C}^*g)_H \\
&=(f,\C_1\mathcal{T}^*g)_H\,.
\end{align*}
An analogous computation shows that also $\C_2\mathcal{T}=\C_2\mathcal{T}^*$, and the second claim of the lemma holds.
Hence, the proof is complete. 
\end{proof}

Since $\C\in L(\mathcal{H})$, it follows that $\C$ generates a uniformly
continuous $C_0$-semigroup $\Sg(t)$, $t\geq 0$, with $\Sg(t)=\exp(t\C)$ (see for example 
Gawarecki and Mandrekar~\cite[Thm.~1.1]{GM}). We note the following for $\C_i$, $i=1,2$:
\begin{lemma}
\label{lem:C-semigroup}
For the $C_0$-semigroup $\Sg_i$ generated by $\C_i$ in \eqref{def-C1} and \eqref{def-C2}, $i=1,2$, resp., we have
$$
\Sg_1(t)\mathcal{T}=\sum_{n=0}^{\infty}\frac{t^n}{n!}\mathcal{C}^n\mathcal{T}(\mathcal{C}^*)^n\,,
$$
and
$$
\Sg_2(t)\mathcal{T}=\exp(t\mathcal{C})\mathcal{T}\exp(t\mathcal{C}^*)\,,
$$
for every $\mathcal{T}\in\mathcal{H}$.
\end{lemma}
\begin{proof}
For $\mathcal{T}\in\mathcal{H}$, we find for $n\geq 1$
$$
\C_1^n\mathcal{T}=\C_1^{n-1}(\mathcal{C}\mathcal{T}\mathcal{C}^*)\,,
$$
and iterating this yields
$$
\C_1^n\mathcal{T}=\mathcal{C}^n\mathcal{T}\mathcal{C}^{*n}\,.
$$
Hence, the result for $\Sg_1$ follows. 

For the case $\C_2$, note that 
$$
\exp(t\mathcal{C})\mathcal{T}\exp(t\mathcal{C}^*)=\sum_{n,m=0}^{\infty}\frac{t^{n+m}}{n!m!}\mathcal{C}^n\mathcal{T}\mathcal{C}^{*m}\,.
$$
On the other hand,
$$
\exp(t\C_2)\mathcal{T}=\sum_{k=0}^{\infty}\frac{t^k}{k!}\C_2^n\mathcal{T}\,.
$$
Spelling out $\C_2^n\mathcal{T}$ and comparing with the terms in the double-sum above, we show the second result.
The proof of the lemma is complete.
\end{proof}

We now introduce the operator-valued BNS stochastic volatility model. To this end,
assume that $\{\Y(t)\}_{t\geq 0}$ is 
a $\mathcal{H}$-valued stochastic process satisfying the dynamics
\begin{equation}
\label{dyn-Y}
d\Y(t)=\C\Y(t)\,dt+d\Le(t)\,\qquad \Y(0)=\Y_0\,.
\end{equation}
Here, $\Le$ is an $\mathcal{H}$-valued L\'evy process and $\Y_0\in\mathcal{H}$. 
We suppose that $\Le$ is square-integrable, with covariance operator $\mathbb{Q}_{\Le}$. 
Recall that $\mathbb{Q}_{\Le}$ is a self-adjoint non-negative definite trace class operator on
$\mathcal{H}$. We have,
\begin{lemma}
\label{lem:integrability}
For every $t\geq 0$, it holds
$$
\int_0^t\|\Sg(t-s)\mathbb{Q}_{\Le}^{1/2}\|_{L_{\text{HS}}(\mathcal{H})}^2\,ds\leq\frac{\tr(\mathbb{Q}_{\Le})}{2\|\C\|_{\text{op}}}(\e^{2t\|\C\|_{\text{op}}}-1)<\infty\,.
$$
\end{lemma}
\begin{proof}
Note first that for any $\mathcal{T}\in\mathcal{H}$, we have by the representation of $\Sg$,
\begin{align*}
\|\Sg(u)\mathcal{T}\|_{\mathcal{H}}&\leq\|\Sg(u)\|_{\text{op}}\|\mathcal{T}\|_{\mathcal{H}} \\
&\leq\|\mathcal{T}\|_{\mathcal{H}}\sum_{k=0}^{\infty}\frac{u^k}{k!}\|\C\|_{\text{op}}^k \\
&=\e^{u\|\C\|_{\text{op}}}\|\mathcal{T}\|_{\mathcal{H}}\,.
\end{align*}
But then, for an orthonormal basis $\{\mathcal{T}_n\}_{n\in\N}\subset\mathcal{H}$\,,
\begin{align*}
\|\Sg(u)\mathbb{Q}_{\Le}^{1/2}\|_{L_{\text{HS}}(\mathcal{H})}^2&=\sum_{n=1}^{\infty}\|\Sg(u)\mathbb{Q}_{\Le}^{1/2}\mathcal{T}_n\|^2_{\mathcal{H}} \\
&\leq \e^{2u\|\C\|_{\text{op}}}\sum_{n=1}^{\infty}\|\mathbb{Q}^{1/2}_{\Le}\mathcal{T}_n\|_{\mathcal{H}}^2 \\
&=\e^{2u\|\C\|_{\text{op}}}\tr(\mathbb{Q}_{\Le})\,.
\end{align*}
Here we have used the fact that 
$$
\tr(\mathbb{Q}_{\Le})=\sum_{n=1}^{\infty}\langle\mathbb{Q}_{\Le}\mathcal{T}_n,\mathcal{T}_n\rangle_{\mathcal{H}}=
\sum_{n=1}^{\infty}\|\mathbb{Q}_{\Le}^{1/2}\mathcal{T}_n\|^2_{\mathcal{H}}\,.
$$
Hence, since $\|\C\|_{\text{op}}<\infty$ and $\mathbb{Q}_{\Le}$ is a trace class operator, the result follows. 
\end{proof}
Invoking this lemma, it follows from the theory of Hilbert-space valued
stochastic differential equations (see {\it e.g.} Peszat and Zabczyk~\cite{PZ}) that there exists a unique mild solution
to \eqref{dyn-Y}
\begin{equation}
\label{sol-Y}
\Y(t)=\Sg(t)\Y_0+\int_0^t\Sg(t-s)\,d\Le(s)\,,
\end{equation}
for $t\geq 0$. In the next lemma we derive a bound for the $L^2$-norm of $\Y$:
\begin{lemma}
\label{lem:norm-bound-Y}
It holds that
$$
\E\left[\|\Y(t)\|_{\mathcal{H}}^2\right]\leq c\e^{2t\|\C\|_{\text{op}}}
$$
for a constant $c$ which is bounded by $\max(2\|\Y_0\|_{\mathcal{H}}^2,\tr(\mathbb{Q}_{\Le})/\|\C\|_{\text{op}})$. 
\end{lemma} 
\begin{proof}
From the mild solution of $\Y(t)$ in \eqref{sol-Y} and the triangle inequality we find,
$$
\E\left[\|\Y(t)\|_{\mathcal{H}}^2\right]\leq 2\|\Sg(t)\Y_0\|_{\mathcal{H}}^2+2\int_0^t\|\Sg(t-s)\mathbb{Q}_{\Le}^{1/2}\|_{L_{\text{HS}}(\mathcal{H})}^2\,ds\,,
$$
where we used Cor.~8.17 in Peszat and Zabczyk~\cite{PZ}. But from Lemma~\ref{lem:integrability} above the result follows.
\end{proof}

Let us compute the conditional characteristic function of $\Y(t)$: To this end, let $t\geq s$ and note that 
$\Y(t)$ given $\Y(s)$ has the representation
\begin{equation}
\label{Y-sol-cond}
\Y(t)=\Sg(t-s)\Y(s)+\int_s^t\Sg(t-u)\,d\Le(u)\,.
\end{equation}
Before proceeding, we recall the cumulant of $\Le$, that is, the characteristic exponent of the L\'evy process $\Le$ defined
to be $\E[\exp(\mathrm{i}\langle\Le(t),\mathcal{T}\rangle_{\mathcal{H}}]=\exp(t\Psi_{\Le}(\mathcal{T}))$ for $\mathcal{T}\in\mathcal{H}$ 
(see Peszat and Zabczyk~\cite[Thm.~4.27]{PZ}):
\begin{equation}
\label{cumulant-L}
\Psi_{\Le}(\mathcal{T})=\mathrm{i}\langle\mathcal{D},\mathcal{T}\rangle_{\mathcal{H}}-\frac12\langle\mathbb{Q}_{\Le}^0\mathcal{T},\mathcal{T}\rangle_{\mathcal{H}}
+\int_{\mathcal{H}}( \e^{\mathrm{i}\langle\mathcal{Z},\mathcal{T}\rangle_{\mathcal{H}}}-1-\mathrm{i}\mathbf{1}_{\|\mathcal{Z}\|_{\mathcal{H}}\leq 1}\langle\mathcal{Z},\mathcal{T}\rangle_{\mathcal{H}})\,\nu(d\mathcal{Z})\,.
\end{equation}
Here, following Peszat and Zabczyk~\cite[Thms.~4.44 and 4.47]{PZ}, $\nu$ is the L\'evy measure on $\mathcal{H}$ and $\mathcal{D}\in\mathcal{H}$ is the drift of the L\'evy process, 
where for $\mathcal{T}\in\mathcal{H}$,
$$
\E[\langle\Le(1),\mathcal{T}\rangle_{\mathcal{H}}]=\langle\mathcal{D},\mathcal{T}\rangle_{\mathcal{H}}+\int_{\mathcal{H}\backslash\{\|\mathcal{Z}\|_{\mathcal{H}}<1\}}
\langle\mathcal{Z},\mathcal{T}\rangle_{\mathcal{H}}\,\nu(d\mathcal{Z})\,.
$$ 
Furthermore, $\mathbb{Q}_{\Le}=\mathbb{Q}_{\Le}^0+\mathbb{Q}_{\Le}^1$ and
$$
\langle\mathbb{Q}^1_{\Le}\mathcal T,\mathcal{U}\rangle_{\mathcal{H}}=\int_{\mathcal{H}}\langle\mathcal{T},\mathcal{Z}\rangle_{\mathcal{H}}\langle\mathcal{U},\mathcal{Z}\rangle_{\mathcal{H}}\,\nu(d\mathcal{Z})\,,\qquad\mathcal{T},\mathcal{U}\in\mathcal{H}\,.
$$
We have the following proposition, showing that $\Y$ is an {\it affine} process in $\mathcal{H}$:
\begin{proposition}
\label{prop:cumulant-Y}
For any $\mathcal{T}\in\mathcal{H}$ it holds that
$$
\E\left[\e^{\mathrm{i}\langle\Y(t),\mathcal{T}\rangle_{\mathcal{H}}}\,|\,\mathcal{F}_s\right]=\exp\left(\mathrm{i}\langle\Y(s),\Sg^*(t-s)\mathcal{T}\rangle_{\mathcal{H}}
+\int_0^{t-s}\Psi_{\Le}(\Sg^*(u)\mathcal{T})\,du\right)\,.
$$
\end{proposition}
\begin{proof}
From \eqref{Y-sol-cond} we find for $\mathcal{T}\in\mathcal{H}$,
\begin{align*}
\E\left[\e^{\mathrm{i}\langle\Y(t),\mathcal{T}\rangle_{\mathcal{H}}}\,|\,\mathcal{F}_s\right]&=\e^{\mathrm{i}\langle\Sg(t-s)\Y(s),\mathcal{T}\rangle_{\mathcal{H}}}
\E\left[\e^{\mathrm{i}\langle\int_s^t\Sg(t-u)\,d\Le(u),\mathcal{T}\rangle_{\mathcal{H}}}\,|\,\mathcal{F}_s\right] \\
&=\e^{\mathrm{i}\langle\Y(s),\Sg^*(t-s)\mathcal{T}\rangle_{\mathcal{H}}}
\E\left[\e^{\mathrm{i}\langle\int_s^t\Sg(t-u)\,d\Le(u),\mathcal{T}\rangle_{\mathcal{H}}}\right]\,. 
\end{align*}
Here, we have appealed to the independent increment property of L\'evy processes. Hence, from Peszat and Zabczyk~\cite[Thm.~4.27]{PZ} it holds that 
$$
\E\left[\e^{\mathrm{i}\langle\int_s^t\Sg(t-u)\,d\Le(u),\mathcal{T}\rangle_{\mathcal{H}}}\right]=\exp\left(\int_0^{t-s}\Psi_{\Le}(\Sg^*(u)\mathcal{T})\,ds\right)\,,
$$ 
with $\Psi_{\Le}$ defined in \eqref{cumulant-L}. The result follows.
\end{proof}

To define a stochastic volatility based on $\Y$ in \eqref{sol-Y} we must impose positivity constraints. This means that
we want to restrict our attention to $\mathcal{Y}$'s which are self-adjoint, non-negative definite Hilbert-Schmidt operators on $H$ for each $t\geq 0$. We now analyse additional conditions on
$\C$ and $\mathcal{L}$ ensuring non-negative definiteness of $\mathcal{Y}$. First, we show that $\Y(t)$ is self-adjoint whenever $\Le(t)$ is
under a mild condition on $\C$:
\begin{proposition}
\label{prop:Y-sym}
Suppose that $(\C\mathcal{T})^*=\C\mathcal{T}^*$ for any $\mathcal{T}\in\mathcal{H}$. If $\{\Le(t)\}_{t\geq 0}$ is a family of 
self-adjoint operators on $H$ and $\Y_0$ is self-adjoint, 
then $\Y(t)$ is a self-adjoint operator on $H$ for every
$t\geq 0$.
\end{proposition}  
\begin{proof}
Let $f,g\in H$. Then we compute, using the dynamics of $\Y$ in \eqref{dyn-Y}, the assumption on $\C$, the self-adjointness of $\Le(t)$ and the definition of Bochner integration:
\begin{align*}
(\Y(t)f,g)_H&=\int_0^t(\C\Y(s)f,g)_H\,ds+(\Le(t)f,g)_H \\
&=\int_0^t(f,\C\Y^*(s)g)_H\,ds+(f,\Le(t)g)_H\,.
\end{align*}
Thus, as $f,g\in H$ are arbitrary, we find that 
$$
d\Y^*(t)=\C\Y^*(t)\,dt+d\Le(t)\,,
$$
with initial condition $\Y^*(0)=\Y_0$. But by uniqueness of solutions of this linear stochastic differential equation, 
$\Y^*(t)=\Y(t)$. 
\end{proof}
Recall from Lemma~\ref{lemma:basic-prop-C} that $(\C_i\mathcal T)^*=\C_i\mathcal T^*$ for $i=1,2$. 
\begin{example}
\label{trivial-ex-levy}
A trivial way to introduce a self-adjoint L\'evy process $\Le$ in $\mathcal{H}$ is to take any real-valued L\'evy
process $L$ and multiply it with a self-adjoint operator $\mathcal{U}\in\mathcal{H}$, i.e., $\Le(t)=L(t)\mathcal{U}$. 
For $\mathcal{S},\mathcal{T}\in\mathcal{H}$,
\begin{align*}
\E\left[\langle\Le(t),\mathcal{S}\rangle_{\mathcal{H}}\langle\Le(t),\mathcal{T}\rangle_{\mathcal{H}}\right]&=\E[L^2(t)]\langle\mathcal{U},\mathcal{S}\rangle_{\mathcal{H}}
\langle\mathcal{U},\mathcal{T}\rangle_{\mathcal{H}}
=t\text{Var}(L(1))\langle\mathcal{U}^{\otimes 2}\mathcal{S},\mathcal{T}\rangle_{\mathcal{H}}\,.
\end{align*}
Thus, the covariance operator for this L\'evy process becomes $\mathbb{Q}_{\Le}=\mathcal{U}^{\otimes 2}$, i.e., the tensor product of $\mathcal{U}$ with itself. We show that $\mathbb{Q}_{\Le}$
is a self-adjoint, non-negative definite trace class operator. Indeed, it is obviously linear and 
$$
\|\mathbb{Q}_{\Le}\mathcal{T}\|_{\mathcal{H}}=\|\mathcal{U}\langle\mathcal{U},\mathcal{T}\rangle_{\mathcal{H}}\|_{\mathcal{H}}\leq \|\mathcal{U}\|^2_{\mathcal{H}}\|\mathcal{T}\|_{\mathcal{H}}\,,
$$ 
which shows $\mathbb{Q}_{\Le}\in L(\mathcal{H})$. Moreover,
$$
\langle\mathbb{Q}_{\Le}\mathcal{S},\mathcal{T}\rangle_{\mathcal{H}}=\langle\mathcal{U},\mathcal{S}\rangle_{\mathcal{H}}\langle\mathcal{U},\mathcal{T}\rangle_{\mathcal{H}}
=\langle\mathcal{S},\mathcal{U}^{\otimes 2}\mathcal{T}\rangle_{\mathcal{H}}=\langle\mathcal{S},\mathbb{Q}_{\Le}\mathcal{T}\rangle_{\mathcal{H}}
$$
and
$$
\langle\mathbb{Q}_{\Le}\mathcal{S},\mathcal{S}\rangle_{\mathcal{H}}=\langle\mathcal{U},\mathcal{S}\rangle_{\mathcal{H}}^2\geq 0\,,
$$
which show that $\mathbb{Q}_{\Le}$ is a self-adjoint and non-negative definite operator on $\mathcal{H}$. Finally, for an orthonormal basis $\{\mathcal{T}_n\}_{n\in\N}$ in $\mathcal{H}$,
$$
\tr(\mathbb{Q}_{\Le})=\sum_{n=1}^{\infty}\langle\mathbb{Q}_{\Le}\mathcal{T}_n,\mathcal{T}_n\rangle_{\mathcal{H}}=\sum_{n=1}^{\infty}\langle\mathcal{U},\mathcal{T}_n\rangle_{\mathcal{H}}^2=\|\mathcal{U}\|_{\mathcal{H}}^2
$$
where we used Parseval's identity. Hence, $\mathbb{Q}_{\Le}$ is trace class. Of course, if we add the assumption that $\mathcal{U}$ is positive
definite and $L(t)$ is taking values on $\R_+$,\footnote{This means that the L\'evy process is a so-called subordinator, that is, a process with 
only positive jumps and non-negative drift.} it follows that 
$$
(\Le(t)f,f)_H=L(t)(\mathcal{U}f,f)_H\geq 0\,,
$$
for any $f\in H$, and thus $\Le(t)$ is non-negative definite. 
\end{example}
This simple example of an operator-valued L\'evy process $\Le$ brings us to the question of non-negative definiteness of $\Y$, which we investigate next.
First, let us define what we mean by non-decreasing paths of $\Le$:
\begin{definition}
We say that the $\mathcal{H}$-valued L\'evy process $\Le$ has {\it non-decreasing paths} if $\Le(t)$ is self-adjoint and 
$t\mapsto(\Le(t)f,f)_H$ is non-decreasing in $t\geq 0$ for every $f\in H$, $a.s$. 
\end{definition}
Note that as $\Le(0)=0$ by definition of the L\'evy process, the non-decreasing paths property implies $(\Le(t)f,f)_H\geq 0$ for every $t\geq 0$, $a.s.$. But then it follows
that $\Le(t)$ is a non-negative definite operator. In fact, something slightly stronger holds:
 \begin{lemma}
\label{lem:pos-def-L}
Assume $\Le$ is an $\mathcal{H}$-valued L\'evy process with non-decreasing paths. Then $\Le(t)-\Le(s)$ is $a.s$ non-negative definite for
every $t>s\geq 0$.
\end{lemma}
\begin{proof}
For $t>s\geq 0$, we have for $f\in H$
$$
((\Le(t)-\Le(s))f,f)_H=(\Le(t)f,f)_H-(\Le(s)f,f)_H
$$
which is non-negative $a.s$ by the non-decreasing path property of $t\mapsto(\Le(t)f,f)_H$. The assertion follows.
\end{proof}
As we shall see, this monotonicity property of the "paths" of $\Le$ is exactly what we need in order to show that $\Y(t)$ is non-negative definite
for every $t\geq 0$.  But first, let us do some analysis of L\'evy processes in $\mathcal{H}$ with non-decreasing paths.

Define for the moment $L_f(t):=(\Le(t)f,f)_H$ for given $f\in H$. We show that this is a L\'evy process on the real line. To this end, consider
the functional $\mathcal{F}_f:\mathcal{H}\rightarrow\R$ defined as $\mathcal{F}_f(\mathcal{T})=(\mathcal{T}f,f)_H$. This is 
obviously linear, and since
$$
|\mathcal{F}_f(\mathcal{T})|=|(\mathcal{T}f,f)_H|\leq|\mathcal{T}f|_H|f|_H\leq\|\mathcal{T}|_{\text{op}}|f|^2_H\,,
$$
we have $\mathcal{F}_f\in\mathcal{H}^*$. Hence, there exists a unique element in $\mathcal{H}$, which we also
denote by $\mathcal{F}_f$, 
$$
\mathcal{F}_f(\mathcal{T})=\langle\mathcal{T},\mathcal{F}_f\rangle_{\mathcal{H}}\,.
$$
In the following we identify Hilbert-Schmidt operators on $H$ with $H\otimes H$. Similarly, the Hilbert-Schmidt operator $h^*\otimes h$ for $h\in H$ is written as $h\otimes h$. Then for any Hilbert-Schmidt operator $\mathcal{V}$ we have the following identity
\[ \langle \mathcal{V},h\otimes h\rangle_{\mathcal{H}} = (\mathcal{V}h,h)_{H}. \]
Thus, we have $\mathcal{F}_f=f\otimes f$. Indeed, a straightforward calculation shows,
\begin{align*}
\|f\otimes f\|_{\mathcal{H}}^2&=\sum_{n=1}^{\infty}((f\otimes f)e_n,(f\otimes f)e_n)_H \\
&=\sum_{n=1}^{\infty}((f,e_n)_Hf,(f,e_n)_Hf)_H \\
&=|f|_H^2\sum_{n=1}^{\infty}(f,e_n)_H^2 \\
&=|f|^4_H\,.
\end{align*}
Hence, $f\otimes f\in\mathcal{H}$ with norm $\|f\otimes f\|_{\mathcal{H}}=|f|_H^2$. Furthermore,
\begin{align*}
\langle\mathcal{T},f\otimes f\rangle_{\mathcal{H}}&=\sum_{n=1}^{\infty}(\mathcal{T}e_n,(f\otimes f)e_n)_H \\
&=\sum_{n=1}^{\infty}(\mathcal{T}e_n,(f,e_n)_Hf)_H \\
&=\sum_{n=1}^{\infty}(\mathcal{T}(f,e_n)_He_n,f)_H \\
&=(\mathcal{T}f,f)_H\,.
\end{align*}
By definition of an $\mathcal{H}$-valued L\'evy process, $t\mapsto\langle\Le(t),\mathcal{T}\rangle_{\mathcal{H}}$ is
a real-valued L\'evy process for any $\mathcal{T}\in\mathcal{H}$. 
Therefore, in particular, $L_f(t)=(\Le(t)f,f)_H$ is a real-valued L\'evy process by choosing $\mathcal{T}=f\otimes f$. If, furthermore, $\Le$ has non-decreasing
paths, it follows that $L_f$ is a L\'evy process with non-decreasing paths, i.e., a subordinator. We have the property of the continuous martingale part of
$\Le$:

\begin{proposition}\label{prop:zeroMGpart}
  Let $\Le$ be the L\'evy process defined after \eqref{dyn-Y} with non-decreasing paths, and denote the covariance operator of the continuous martingale part by $\mathbb{Q}_{\Le}^0$. Let $\mathcal{T}$ be a symmetric Hilbert-Schmidt operator. Then $\mathbb{Q}_{\Le}^0\mathcal{T} = 0$, that is, 
$\mathcal{T}\in\ker(\mathbb{Q}_{\Le}^0)$.
\end{proposition}
\begin{proof}
Let first $\mathcal{T} = f\otimes f$ with $f\in H$. Then the continuous martingale part of the characteristic function of 
$L_f(t)=\langle\Le(t),\mathcal{F}_f\rangle_{\mathcal{H}}$ is $\langle\mathbb{Q}_{\Le}^0\mathcal{F}_f,\mathcal{F}_f\rangle_{\mathcal{H}}$, which must be zero
due to the non-decreasing paths of $L_f(t)$. But then
$$
\langle\mathbb{Q}_{\Le}^0\mathcal{F}_f,\mathcal{F}_f\rangle_{\mathcal{H}}=\|(\mathbb{Q}_{\Le}^0)^{1/2}\mathcal{F}_f\|_{\mathcal{H}}^2=0\,,
$$
and thus $\mathcal{F}_f$ is in the kernel of $(\mathbb{Q}_{\Le}^0)^{1/2}$. As it holds,
$$
\mathbb{Q}_{\Le}^0\mathcal{F}_f=(\mathbb{Q}_{\Le}^0)^{1/2}(\mathbb{Q}_{\Le}^0)^{1/2}\mathcal{F}_f=(\mathbb{Q}_{\Le}^0)^{1/2}0=0\,,
$$
we can conclude that $\mathcal{F}_f\in\ker(\mathbb{Q}_{\Le}^0)$.

Now let $\mathcal{T}$ be a symmetric Hilbert-Schmidt operator as in the proposition. It can be shown that $\mathcal{T}$ must be of the form
\begin{align*}
  \mathcal{T} = \sum_{k,l\in\N} \gamma_{k,l} e_k\otimes e_l,
\end{align*}
with $\sum_{k,l} \gamma_{k,l}^2<\infty$ and $\gamma_{k,l} = \gamma_{l,k}$, see Lemma \ref{lem:app1} for a sketch of the arguments. Therefore we can write
\begin{align*}
	\mathcal{T}	& = \sum_{k\in\N} \gamma_{k,k} e_k\otimes e_k + 2\sum_{k\in\N}\sum_{l<k} \gamma_{k,l} (e_k\otimes e_l + e_l\otimes e_k) \\
		& = \sum_{k\in\N} \gamma_{k,k} e_k\otimes e_k + 2\sum_{k\in\N}\sum_{l<k} \gamma_{k,l} \big((e_k+e_l)\otimes(e_k\otimes e_l) - e_k\otimes e_k - e_l\otimes e_l\big).
\end{align*}
With this we compute
\begin{align*}
 \mathbb{Q}_{\Le}^0\mathcal{T} &= \sum_{k\in\N} \gamma_{k,k}\mathbb{Q}_{\Le}^0(e_k\otimes e_k) \\
&\qquad+ 2\sum_{k\in\N}\sum_{l<k}\gamma_{k,l} \big( \mathbb{Q}_{\Le}^0((e_k+e_l)\otimes (e_k + e_l)) - \mathbb{Q}_{\Le}^0(e_k\otimes e_k) - \mathbb{Q}_{\Le}^0 (e_l\otimes e_l)\big),
\end{align*}
which ends the proof since $\mathbb{Q}_{\Le}^0$ applied to a symmetric operator is zero by the first part of the proof.
\end{proof}

As the space of symmetric Hilbert-Schmidt operators does not span
$\mathcal{H}$, we cannot conclude that $\mathbb{Q}_{\Le}^0=0$, i.e., that $\Le$ does not have a continuous martingale part. 
Recall that subordinators on $\R$ do not have any continuous martingale part.

\medskip

Denote now by $\mathcal{H}_+$ the convex cone of non-negative definite operators in $\mathcal{H}$. 
\begin{proposition}
\label{prop:Y-pos-def}
Assume that $\C(\mathcal{H}_+)\subset\mathcal{H}_+$. If $\Le(t)$ is an $\mathcal{H}$-valued L\'evy process with non-decreasing paths and 
$\Y_0$ is non-negative definite, then $\Y$ is non-negative definite. 
\end{proposition}
\begin{proof}
Recall that
$$
\Y(t)=\Sg(t)\Y_0+\int_0^t\Sg(t-s)\,\Le(s)\,.
$$
It holds,
$$
\Sg(t)\Y_0=\e^{t\C}\Y_0=\sum_{k=0}^{\infty}\frac{t^k}{k!}\mathcal{\C}^k\Y_0\,,
$$
which is then a non-negative definite operator by the assumption on $\C$. 

Next, we know that $\int_0^t\Sg(t-s)\,d\Le(s)$ is defined as the strong limit of 
$\sum_{m=1}^{M}\Sg(t-s_m)\,\Delta\Le(s_m)$ in $L^2(\Omega\times[0,t];\mathcal{H})$. Here, $\{s_m\}_{m=1}^{M}$ is a 
nested partition of $[0,t]$, and $\Delta\Le(s_m):=\Le(s_{m+1})-\Le(s_m)$ is an increment of $\Le$. 
But  $\Delta\Le(s_m)$ is non-negative definite $a.s.$ by Lemma~\ref{lem:pos-def-L}, and therefore each term in the sum above is positive,
$a.s.$, since $\Sg$ preserves non-negative definiteness by assumption on $\C$. Hence it follows that $\int_0^t\Sg(t-s)\,d\Le(s)$ is 
non-negative definite $a.s.$,
and the proof is complete.
\end{proof}
From Proposition  \ref{prop:Y-sym} and Proposition \ref{prop:Y-pos-def} it follows that under the assumptions
\begin{enumerate} 
\item[a)] $(\C\mathcal{T})^*=\C\mathcal{T}^*$, 
\item[b)] $\C(\mathcal{H}_+)\subset\mathcal{H}_+$, and
\item[c)] $\Le(t)$ is a self-adjoint and non-negative definite square-integrable L\'evy process with 
 values in $\mathcal{H}$,
\end{enumerate}
then $\Y(t)$ becomes a self-adjoint,
non-negative definite square integrable process with values in $\mathcal{H}$ as long as the initial condition $\Y_0$ is self-adjoint
and non-negative definite. Hence, we have a unique square root, $\Y^{1/2}(t)$ for every $t\geq 0$. We shall
use this to model the stochastic volatility. 

\begin{lemma}
It holds that $\C_1(\mathcal{H}_+)\subset\mathcal{H}_+$.
\end{lemma}
\begin{proof}
We recall the definition of $\C_1$ in \eqref{def-C1}. Let $\mathcal{T}\in\mathcal{H}_+$. Then, for any $f\in H$
$$
(\C_1\mathcal{T}f,f)_H=(\mathcal{C}\mathcal{T}\mathcal{C}^*f,f)_H=(\mathcal{T}\mathcal{C}^*f,\mathcal{C}^*f)_H\geq 0\,.
$$
Hence, the result follows.
\end{proof}
In fact, for $\C_2$ we cannot prove that it preserves the property of non-negative definiteness. But recalling the proof of Prop.~\ref{prop:Y-pos-def}, it is indeed the 
associated semigroup of $\C$ that must be non-negative definite. As we have that $\Sg_2(t)\mathcal{T}=\exp(t\mathcal{C})\mathcal{T}\exp(t\mathcal{C}^*)$ from
Lemma~\ref{lem:C-semigroup}, it follows that $\Sg_2(t)(\mathcal{H}_+)\subset\mathcal{H}_+$, and we can conclude that $\Y$ with $\C=\C_2$ is also
non-negative definite whenever $\Le$ has non-decreasing paths and $\Y_0$ is non-negative definite. Indeed, by inspection of the proof
of Prop.~\ref{prop:Y-pos-def}, we can substitute the condition b) $\C(\mathcal{H}_+)\subset\mathcal H_+$ on $\C$ with the 
condition 
\begin{enumerate}
\item[b')] $\quad \Sg(t)(\mathcal{H}_+)\subset\mathcal{H}_+, \quad t\geq 0\,.$ 
\end{enumerate}
In conclusion, if we use $\C=\C_i$ for either $i=1$ or $i=2$ in the definition of the volatility process $\Y$, we obtain a non-negative definite operator under appropriate conditions
on $\Le$ and $\Y_0$. We recall that the choice $\C=\C_2$ can be seen as the analogue of the matrix-valued volatility model by Barndorff-Nielsen and Stelzer~\cite{BNStelzer}.

\medskip
Let us discuss the particular case when $\Le$ is a compound Poisson process. To this end, we define 
\begin{equation}\label{eqcP}
\Le(t)=\sum_{i=1}^{N(t)}\X_i\,,
\end{equation} 
where $N$ is a real-valued Poisson process with intensity $\lambda>0$ and $\{\X_i\}_{i\in\mathbb{N}}$ are $i.i.d.$ square-integrable 
$\mathcal{H}$-valued random variables.  Note that for $f\in H$, we find from the linearity of the inner product
\begin{align*}
\langle\Le(t),f\otimes f\rangle_{\mathcal{H}}&=\langle\sum_{i=1}^{N(t)}\X_i,f\otimes f\rangle_{\mathcal{H}} 
=\sum_{i=1}^{N(t)}\langle\X_i,f\otimes f\rangle_{\mathcal{H}} 
=\sum_{i=1}^{N(t)}(\X_if,f)_H\,.
\end{align*}
Hence, $L_f(t):=\langle\Le(t),f\otimes f\rangle_{\mathcal{H}}$ is a real-valued compound Poisson process with jumps
given by the $i.i.d$ random variables $(\X_i f,f)_H$. The process $L_f(t)$ has non-decreasing paths if and only if $\X$ is self-adjoint and 
$(\X f,f)_H$ is distributed on $\R_+$, where the latter holds if and only if $\X$ is non-negative definite, i.e., $\X\in\mathcal{H}_+$. Next, introduce the map
$\phi_f:\mathcal{H}_+\rightarrow\R_+$ by
$$
\phi_f(\mathcal{Z})=\langle\mathcal{Z},f\otimes f\rangle_{\mathcal{H}}\,.
$$ 
For any Borel set $A\subset\R_+$, we define $P_{\phi_f}(A):=P_{\X}(\phi_f^{-1}(A))$ where $P_{\X}$ is the law of $\X$. But then
$$
\int_{\R_+}\e^{\mathrm{i}uz}P_{\phi_f}(dz)=\int_{\mathcal{H}_+}(\e^{\mathrm{i}u\cdot}\circ\phi_f)(\mathcal{Z})P_{\X}(d\mathcal{Z})
=\int_{\mathcal{H}_+}\e^{\mathrm{i}u\langle\mathcal{Z},f\otimes f\rangle_{\mathcal{H}}}P_{\X}(d\mathcal{Z})\,,
$$ 
and $P_{X_f}(A)=P_{\X}(\phi_f^{-1}(A))$ with $P_{X_f}$ being the law of $X_f:=(\X f,f)_H$. 

Suppose that $Z$ is an $H$-valued centred square-integrable Gaussian random variable with covariance operator
$\mathcal{Q}_Z$. Let $\mathcal{X}_i=Z_i^{\otimes 2}, i=1,2,\ldots,$, where $\{Z_i\}_{i\in\mathbb{N}}$ are independent
copies of $Z$. First, it is simple to see  
$Z^{\otimes2}$ is also a Hilbert-Schmidt operator, that is $Z^{\otimes 2}\in\mathcal{H}$, since
$$
\|Z^{\otimes 2}\|_{\mathcal{H}}=\sum_{n=1}^{\infty}|Z^{\otimes 2}e_n|_H^2=\sum_{n=1}^{\infty}(Z,e_n)_H^2|Z|^2_H=|Z|^4_H<\infty\,.
$$
This $\mathcal{H}$-valued random variable has expected ($\mathcal{H}$-valued) value $\E[Z^{\otimes 2}]=\mathcal{Q}_Z$, which can
be seen from the following calculation: given $\mathcal{T}\in\mathcal{H}$, then by linearity of the expectation operator
\begin{align*}
\langle\E\left[Z^{\otimes 2}\right],\mathcal{T}\rangle_{\mathcal{H}}&=\sum_{n=1}^{\infty}(\E[Z^{\otimes 2}]e_n,\mathcal{T}e_n)_H \\
&=\sum_{n=1}^{\infty}\E\left[(Z^{\otimes 2}e_n,\mathcal{T}e_n)_H\right] \\
&=\sum_{n=1}^{\infty}\E\left[(Z,e_n)_H(Z,\mathcal{T}e_n)_H\right] \\
&=\sum_{n=1}^{\infty}(\mathcal{Q}_Ze_n,\mathcal{T}e_n)_H \\
&=\langle \mathcal{Q}_Z,\mathcal{T}\rangle_{\mathcal{H}}\,.
\end{align*}
Furthermore, $Z^{\otimes 2}$ is self-adjoint and non-negative definite, since $(Z^{\otimes 2}f,f)_H=(Z,f)^2\geq 0$. From this we also see that the jumps
of $L_f(t)$, the compound Poisson process $\Le$ evaluated at $f\otimes f$, is given by $(Z,f)_H^2$, with $(Z,f)_H$ being a real valued 
centred Gaussian variable with variance $|\mathcal{Q}_Z^{1/2}f|^2_H$. Hence, $(Z,f)_H^2$ becomes Gamma distributed with
scale parameter $2|\mathcal{Q}_Z^{1/2}f|^2_H$ and shape parameter $1/2$. In fact, something much more general can be said about the
compound Poisson process $\Le$ for jumps given by $\mathcal{X}=Z^{\otimes 2}$. Indeed, if $\mathcal{T}\in\mathcal{H}$ is
self-adjoint, then it follows from Prop.~3 in Fraisse and Viguier-Pla~\cite{FVP} that the characteristic functional of 
$\langle Z^{\otimes 2},\mathcal{T}\rangle_{\mathcal{H}}$ is,
\begin{equation}
\E\left[\exp(\mathrm{i}\langle Z^{\otimes 2},\mathcal{T}\rangle_{\mathcal{H}})\right]=\left(\text{det}(I-2\mathrm{i}\mathcal{T}\mathcal{Q}_Z)\right)^{-1/2}\,.
\end{equation}  
Here, $I$ is the identity operator on $H$ and $\text{det}$ is the Fredholm determinant. We can interpret  $Z^{\otimes 2}$ as
being infinite dimensional Wishart distributed. By conditioning of $N(t)$ and appealing
to the independence of the jumps $\mathcal{X}_i$, we find the cumulant $\Psi_{\Le}$ of $\Le$ defined in $\eqref{cumulant-L}$ to be
\begin{equation}
\Psi_{\Le}(\mathcal{T})=\ln\E\left[\exp(\mathrm{i}\langle\Le(1),\mathcal{T}\rangle_{\mathcal{H}})\right]=
\lambda\left((\text{det}(I-2\mathrm{i}\mathcal{T}\mathcal{Q}_Z))^{-1/2}-1\right)\,,
\end{equation}
for any self-adjoint $\mathcal{T}\in\mathcal{H}$. 

Suppose now more in general  that  $Z$ is an $H$-valued centred square-integrable random variable. Then $Z$ has  a self-adjoint non-negative definite continuous linear  covariance operator
$\mathcal{Q}_Z$, too. Let $\mathcal{X}_i=Z_i^{\otimes 2}, i=1,2,\ldots,$, where $\{Z_i\}_{i\in\mathbb{N}}$ are independent copies of $Z$. Then by the same calculations as before $\|Z^{\otimes 2}\|_{\mathcal{H}}=|Z|^4_H<\infty$ and 
$\langle\E\left[Z^{\otimes 2}\right],\mathcal{T}\rangle_{\mathcal{H}}=\langle \mathcal{Q}_Z,\mathcal{T}\rangle_{\mathcal{H}}$. Also here $Z^{\otimes 2}$ is self-adjoint and non-negative definite, since $(Z^{\otimes 2}f,f)_H=(Z,f)^2\geq 0$ and  the jumps
of $L_f(t)$, the compound Poisson process $\Le$ evaluated at $f\otimes f$, is given by $(Z,f)_H^2$, with $(Z,f)_H$ being a real valued 
centred  variable with variance $|\mathcal{Q}_Z^{1/2}f|^2_H$.  The cumulant has then to be computed for each case separately.

\section{A volatility-modulated Ornstein-Uhlenbeck process}

Let $X$ be a stochastic process with values in $H$ given by the Ornstein-Uhlenbeck process
\begin{equation}
\label{X-dyn}
dX(t)=\mathcal{A}X(t)\,dt+\Y^{1/2}(t)\,dB(t)\,\qquad X(0)=X_0\,.
\end{equation}
Here, $B$ is a $H$-valued Wiener process with covariance operator $\mathcal{Q}$, which is a self-adjoint,
non-negative definite trace class operator on $H$. Furthermore, $X_0\in H$ and $\Y$ is given in \eqref{sol-Y}, being the 
solution of the dynamics \eqref{dyn-Y} from the previous section, 
where we assume that $\Y_0$ is self-adjoint, non-negative definite and $\Le$ is a $\mathcal{H}$-valued
L\'evy process with non-decreasing paths. We suppose that $(\C\mathcal T)^*=\C\mathcal T^*$ for every 
$\mathcal T\in\mathcal H$ and $\C(\mathcal{H}_+)\subset\mathcal{H}_+$ (or,
that the semigroup $\Sg(t)$ of $\C$ has this property). Then by Props.~\ref{prop:Y-sym} and \ref{prop:Y-pos-def}, $\Y(t)$ is
self-adjoint, non-negative definite, and we can define its square root $\Y^{1/2}(t)$.  Finally,
$\mathcal{A}$ is a (possibly unbounded) linear operator on $H$, densely defined, generating a $C_0$-semigroup $\mathcal{S}$. 

Let us first show that the stochastic integral in \eqref{X-dyn} makes sense. The following proposition is crucial:
\begin{proposition}
\label{prop:trace-QY}
For every $t\geq 0$, it holds that 
$$
\E\left[\tr(\mathcal{Q}^{1/2}\Y(t)\mathcal{Q}^{1/2})\right]=\tr(\mathcal{Q}^{1/2}\Sg(t)\Y_0\mathcal{Q}^{1/2})+\tr(\mathcal{Q}^{1/2}\int_0^t\Sg(s)\,ds\E[\Le(1)]\mathcal{Q}^{1/2})
$$
where $\int_0^t\Sg(s)\,ds$ is the Bochner integral of $s\mapsto\Sg(s)\in L_{HS}(\mathcal{H})$ and $\E[\Le(1)]$ is the operator-valued expected value of $\Le(1)$. 
\end{proposition}
\begin{proof}
First, note that the trace is linear, to give
$$
\E\left[\tr(\mathcal{Q}^{1/2}\Y(t)\mathcal{Q}^{1/2})\right]=\tr(\mathcal{Q}^{1/2}\Sg(t)\Y_0\mathcal{Q}^{1/2})+\E\left[\tr(\mathcal{Q}^{1/2}\int_0^t\Sg(t-s)\,d\Le(s)\mathcal{Q}^{1/2})
\right]\,.
$$
Suppose for a moment that $\mathcal{X}$ is a $\mathcal{H}$-valued integrable random variable. Then
\begin{align*}
\E\left[\tr(\mathcal{Q}^{1/2}\mathcal{X}\mathcal{Q}^{1/2})\right]&=\sum_{n=1}^{\infty}\E\left[(\mathcal{Q}^{1/2}\mathcal{X}\mathcal{Q}^{1/2}e_n,e_n)_H\right] \\
&=\sum_{n=1}^{\infty}\E\left[(\mathcal{X}\mathcal{Q}^{1/2}e_n,\mathcal{Q}^{1/2}e_n)_H\right] \,.
\end{align*}
But $(\mathcal{X}f,f)_H=\langle\mathcal{X},f\otimes f\rangle_{\mathcal{H}}$, which holds due to the isometry of the Hilbert-Schmidt operators with tensor products of Hilbert spaces, and 
$$
\E[(\mathcal{X}f,f)_H]=\E[\langle\mathcal{X},f\otimes f\rangle_{\mathcal{H}}]=\langle\mathcal{M},f\otimes f\rangle_{\mathcal{H}} \,,
$$
for some $\mathcal{M}\in\mathcal{H}$. This operator is called the mean of $\mathcal{X}$, and we write $\E[\mathcal{X}]=\mathcal{M}$, the operator-valued
expectation of $\mathcal{X}$. 
Thus
\begin{align*}
\E\left[\tr(\mathcal{Q}^{1/2}\mathcal{X}\mathcal{Q}^{1/2})\right]&=\sum_{n=1}^{\infty}\E\left[\langle\mathcal{X},\mathcal{Q}^{1/2}e_n\otimes\mathcal{Q}^{1/2}e_n\rangle_{\mathcal{H}}\right] \\
&=\sum_{n=1}^{\infty}\langle\mathcal{M},\mathcal{Q}^{1/2}e_n\otimes\mathcal{Q}^{1/2}e_n\rangle_{\mathcal{H}} \\
&=\sum_{n=1}^{\infty}(\mathcal{M}\mathcal{Q}^{1/2}e_n,\mathcal{Q}^{1/2}e_n)_H \\
&=\tr(\mathcal{Q}^{1/2}\mathcal{M}\mathcal{Q}^{1/2}) \\
&=\tr(\mathcal{Q}^{1/2}\E[\mathcal{X}]\mathcal{Q}^{1/2})\,.
\end{align*}
Letting $\mathcal{X}=\int_0^t\Sg(t-s)\,d\Le(s)$, we hence obtain
$$
\E[\tr(\mathcal{Q}^{1/2}\int_0^t\Sg(t-s)\,d\Le(s)\mathcal{Q}^{1/2})]=\tr\left(\mathcal{Q}^{1/2}\E\left[\int_0^t\Sg(t-s)\,d\Le(s)\right]\mathcal{Q}^{1/2}\right)
$$
We derive an expression for the mean of the stochastic integral. 

Recalling (the proof of) Prop.~\ref{prop:cumulant-Y}, we find that with $\theta\in\R$
\begin{align*}
\E\left[\e^{\mathrm{i}\langle\int_0^t\Sg(t-s)\,d\Le(s),\theta\mathcal{T}\rangle_{\mathcal{H}}}\right]&=\exp\left(\int_0^t\Psi_{\Le}(\Sg^*(u)(\theta\mathcal{T}))\,du\right) \\
&=\exp\left(\int_0^t\Psi_{\Le}(\theta\Sg^*(u)(\mathcal{T}))\,du\right) \,,
\end{align*}
with $\Psi_{\Le}$ defined in \eqref{cumulant-L}.
Since $\Psi_{\Le}(0)=0$, we find
$$
\frac{d}{d\theta}\E\left[\e^{\mathrm{i}\langle\int_0^t\Sg(t-s)\,d\Le(s),\theta\mathcal{T}\rangle_{\mathcal{H}}}\right]\vert_{\theta=0}=\int_0^t\frac{d}{d\theta}
\Psi_{\Le}(\theta\Sg^*(u)\mathcal{T})\,du\vert_{\theta=0}\,.
$$
But, for any $\mathcal{S}\in\mathcal{H}$,
$$
\frac{d}{d\theta}\Psi_{\Le}(\theta\mathcal{S})=\mathrm{i}\langle\mathcal{D},\mathcal{S}\rangle_{\mathcal{H}}-
\theta\langle\mathbb{Q}_{\Le}^0\mathcal{S},\mathcal{S}\rangle_{\mathcal{H}}+\mathrm{i}\int_{\mathcal{H}}(\langle\mathcal{Z},\mathcal{S}\rangle_{\mathcal{H}}
\e^{\mathrm{i}\theta\langle\mathcal{Z},\mathcal{S}\rangle_{\mathcal{H}}}-\langle\mathcal{Z},\mathcal{S}\rangle_{\mathcal{H}}\mathrm{1}(\|\mathcal{Z}\|_{\mathcal{H}}<1))\,\nu(d\mathcal{Z})\,.
$$
Therefore,
\begin{align*}
\E\left[\langle\int_0^t\Sg(t-s)\,d\Le(s),\mathcal{T}\rangle_{\mathcal{H}}\right]&=(-\mathrm{i})\int_0^t\frac{d}{d\theta}\Psi_{\Le}(\theta\Sg^*(u)\mathcal{T})\vert_{\theta=0}\,du \\
&=\int_0^t\langle\mathcal{D},\Sg^*(u)\mathcal{T}\rangle_{\mathcal{H}}+\int_{\|\mathcal{Z}\|_{\mathcal{H}}>1}\langle\mathcal{Z},\Sg^*(u)\mathcal{T}\rangle_{\mathcal{H}}\,\nu(d\mathcal{Z})\,du \\
&=\int_0^t\langle\Sg(u)\mathcal{D},\mathcal{T}\rangle_{\mathcal{H}}+\int_{\|\mathcal{Z}\|_{\mathcal{H}}>1}\langle\Sg(u)\mathcal{Z},\mathcal{T}\rangle_{\mathcal{H}}\,\nu(d\mathcal{Z})\,du \\
&=\langle\int_0^t\Sg(u)(\mathcal{D}+\int_{\|\mathcal{Z}\|_{\mathcal{H}}>1}\mathcal{Z}\,\nu(d\mathcal{Z})\,du),\mathcal{T} \rangle_{\mathcal{H}} \\
&=\langle\int_0^t\Sg(u)\,du(\mathcal{D}+\int_{\|\mathcal{Z}\|_{\mathcal{H}}>1}\mathcal{Z}\,\nu(d\mathcal{Z}),\mathcal{T} \rangle_{\mathcal{H}}\,.
\end{align*}
Thus, since $\E[\Le(1)]=\mathcal{D}+\int_{\|\mathcal{Z}\|_{\mathcal{H}}>1}\mathcal{Z}\,\nu(d\mathcal{Z})$, we get
$$
\E\left[\int_0^t\Sg(t-s)\,d\Le(s)\right]=\int_0^t\Sg(u)\,du\E[\Le(1)]\,.
$$
This completes the proof.
\end{proof}
To have the stochastic integral $\int_0^t\Y^{1/2}(s)\,dB(s)$ well-defined, the integrand must satisfy the condition
\begin{equation}
\E\left[\int_0^t\|\Y^{1/2}(s)\mathcal{Q}^{1/2}\|^2_{\mathcal{H}}\,ds\right]<\infty\,.
\end{equation}
But $\|\Y^{1/2}(s)\mathcal{Q}^{1/2}\|^2_{\mathcal{H}}=\tr(\mathcal{Q}^{1/2}\Y(s)\mathcal{Q}^{1/2})$. From 
Prop.~\ref{prop:trace-QY} above
we see that the expected value of this trace is integrable in time on any compact set. Thus, $\Y^{1/2}$ can be used as a stochastic volatility operator in the dynamics
of $X$ in \eqref{X-dyn}.

If the stochastic integral $\int_0^t\mathcal{S}(t-s)\Y^{1/2}(s)\,dB(s)$ exists, then we have the mild solution of \eqref{X-dyn}
\begin{equation}
\label{X-sol}
X(t)=\mathcal{S}(t)X_0+\int_0^t\mathcal{S}(t-s)\Y^{1/2}(s)\,dB(s)\,,
\end{equation}
for a given initial condition $X(0)=X_0\in H$. The stochastic integral is well-defined since
$$
\|\mathcal{S}(t-s)\Y^{1/2}(s)\mathcal{Q}^{1/2}\|_{\mathcal{H}}\leq\|\mathcal{S}(t-s)\|_{\text{op}}\|\Y^{1/2}(s)\mathcal{Q}^{1/2}\|_{\mathcal{H}}\,.
$$
By Yosida~\cite{Yosida}, the operator norm of the semigroup $\mathcal{S}$ is at most exponentially growing. Hence, in view of Prop.~\ref{prop:trace-QY},
integrability holds.

Here is a result on the characteristic function of the process $X(t)$:
\begin{proposition}
\label{prop:cumulant-X}
Suppose that there exists a self-adjoint, positive definite operator $\mathcal{D}\in L(H)$ such that $\Y^{1/2}(s)\mathcal{Q}\Y^{1/2}(s)=\mathcal{D}^{1/2}\Y(s)\mathcal{D}^{1/2}$ 
for all $s\geq 0$. Then, if $\Le$ is independent of $B$,
\begin{align*}
\E\left[\e^{\mathrm{i}(X(t),f)_H}\right]&=\exp\left(\mathrm{i}(X_0,\mathcal{S}^*(t)f)_H-\frac12\langle\Y_0,\int_0^t\Sg^*(s)((\mathcal{D}^{1/2}\mathcal{S}^*(t-s)f)\otimes(\mathcal{D}^{1/2}\mathcal{S}^*(t-s)f))\,ds\rangle_{\mathcal{H}}\right) \\
&\qquad\times\exp\left(\int_0^t\Psi_{\Le}\left(-\frac12\int_0^s\Sg^*(s-u)(\mathcal{D}^{1/2}\mathcal{S}^*(u)f\otimes
\mathcal{D}^{1/2}\mathcal{S}^*(u)f)\,du\right)\,ds\right)\,,
\end{align*}
for any $f\in H$. 
\end{proposition}
\begin{proof}
First, from the mild solution of $X(t)$ we find for $f\in H$
$$
(X(t),f)_H=(\mathcal{S}(t)X_0,f)_H+(\int_0^t\mathcal{S}(t-s)\Y^{1/2}(s)\,dB(s),f)_H\,.
$$
We compute the characteristic function of the random variable $(\int_0^t\mathcal{S}(t-s)\Y^{1/2}(s)\,dB(s),f)_H$:
Since $\Le$ and $B$ are independent, we have that $\Y$ and $B$ are independent. From the tower property of conditional 
expectation, we therefore get after conditioning on the $\sigma$-algebra generated by the paths of 
$\Y$:
\begin{align*}
&\E\left[\exp\left(\mathrm{i}(\int_0^t\mathcal{S}(t-s)\Y^{1/2}(s)\,dB(s),f)_H\right)\right] \\
&\qquad=\E\left[\exp\left(-\frac12\int_0^t(\mathcal{Q}\Y^{1/2}(s)\mathcal{S}^*(t-s)f,
\Y^{1/2}(s)\mathcal{S}^*(t-s)f)_H\,ds\right)\right]\,.
\end{align*}
From the property of the operator $\mathcal{D}$, 
\begin{align*}
(\mathcal{Q}\Y^{1/2}\mathcal{S}^*(t-s)f,\Y^{1/2}(s)\mathcal{S}^*(t-s)f)_H&=(\Y^{1/2}(s)\mathcal{Q}\Y^{1/2}(s)\mathcal{S}^*(t-s)f,\mathcal{S}^*(t-s)f)_H \\
&=(\mathcal{D}^{1/2}\Y(s)\mathcal{D}^{1/2}\mathcal{S}^*(t-s)f,\mathcal{S}^*(t-s)f)_H \\
&=(\Y(s)\mathcal{D}^{1/2}\mathcal{S}^*(t-s)f,\mathcal{D}^{1/2}\mathcal{S}^*(t-s)f)_H \\
&=\langle\Y(s),(\mathcal{D}^{1/2}\mathcal{S}^*(t-s)f)\otimes(\mathcal{D}^{1/2}\mathcal{S}^*(t-s)f)\rangle_{\mathcal{H}}\,.
\end{align*}
For simplicity, introduce for the moment the notation $\mathcal{T}(s)\in\mathcal{H}$ for the family of operators parametrized by time $s\geq 0$, defined by
$$
\mathcal{T}(s)=(\mathcal{D}^{1/2}\mathcal{S}^*(s)f)\otimes(\mathcal{D}^{1/2}\mathcal{S}^*(s)f)\,.
$$
Thus, from the mild solution of $\Y$, 
$$
\int_0^t\langle\Y(s),\mathcal{T}(t-s)\rangle_{\mathcal{H}}\,ds=\int_0^t\langle\Sg(s)\Y_0,\mathcal{T}(t-s)\rangle_{\mathcal{H}}\,ds+
\int_0^t\langle\int_0^s\Sg(s-u)\,d\Le(u),\mathcal{T}(t-s)\rangle_{\mathcal{H}}\,ds
$$
We have that 
$$
\int_0^t\langle\Sg(s)\Y_0,\mathcal{T}(t-s)\rangle_{\mathcal{H}}\,ds=\langle\Y_0,\int_0^t\Sg^*(s)\mathcal{T}(t-s)\,ds\rangle_{\mathcal{H}}
$$
where the integral on the right-hand side is interpreted in the Bochner sense. 
It holds, after appealing to a Fubini theorem for stochastic integrals in Hilbert space (see Peszat and Zabczyk~\cite[Theorem 8.14]{PZ})
\begin{align*}
\int_0^t\langle\int_0^s\Sg(s-u)\,d\Le(u),\mathcal{T}(t-s)\rangle_{\mathcal{H}}\,ds&=\int_0^t\langle\int_u^t\Sg^*(s-u)\mathcal{T}(t-s)\,ds,d\Le(u)\rangle_{\mathcal{H}}\,.
\end{align*}
The $ds$-integral inside the inner product is again interpreted as a Bochner integral.  Hence,
\begin{align*}
&\E\left[\exp\left(-\frac12\int_0^t\langle\int_0^s\Sg(s-u)\,d\Le(u),\mathcal{T}(t-s)\rangle_{\mathcal{H}}\,ds\right)\right] \\
&\qquad\qquad=\E\left[\exp\left(-\frac12\int_0^t\langle\int_u^t\Sg^*(s-u)\mathcal{T}(t-s)\,ds,d\Le(u)\rangle_{\mathcal{H}}\right)\right]\\
&\qquad\qquad=\exp\left(\int_0^t\Psi_{\Le}\left(-\frac12\int_u^t\Sg^*(s-u)(\mathcal{D}^{1/2}\mathcal{S}^*(t-s)f\otimes
\mathcal{D}^{1/2}\mathcal{S}^*(t-s)f)\,ds\right)\,du\right) \\
&\qquad\qquad=\exp\left(\int_0^t\Psi_{\Le}\left(-\frac12\int_0^s\Sg^*(s-u)(\mathcal{D}^{1/2}\mathcal{S}^*(u)f\otimes
\mathcal{D}^{1/2}\mathcal{S}^*(u)f)\,du\right)\,ds\right)\,. 
\end{align*}
This proves the Proposition.
\end{proof}
The result shows that we recover an affine structure of $X$ in terms of $X_0$ and $\Y_0$.
Note that if $\mathcal{Q}$ commutes with $\Y(s)$, then 
$\mathcal{Q}^{1/2}$ commutes with $\Y^{1/2}(s)$, and we find
$$
\Y^{1/2}(s)\mathcal{Q}\Y^{1/2}(s)=\mathcal{Q}^{1/2}\Y(s)\mathcal{Q}^{1/2}\,.
$$
Hence, in this case $\mathcal{D}=\mathcal{Q}$. Indeed, this puts rather strong restrictions on the volatility model $\Y$.
A sufficient condition for $\Y$ commuting with $\mathcal{Q}$ is that $\mathcal{Q}$ commutes with $\Y_0$ and $\Le(t)$ for all
$t\geq 0$, and that $\C(\mathcal{T})\mathcal{Q}=\C(\mathcal{T}\mathcal{Q})$ and 
$\mathcal{Q}\C(\mathcal{T})=\C(\mathcal{Q}\mathcal{T})$ for every $\mathcal{T}\in\mathcal{H}$. If this is the case, we have from
the dynamics of $\Y$ in \eqref{dyn-Y}
\begin{equation}
\label{dyn-YQ1}
\mathcal{Q}\Y(t)=\mathcal{Q}\Y_0+\int_0^t\C(\mathcal{Q}\Y(s))\,ds+\mathcal{Q}\Le(t)\,,
\end{equation}
and
\begin{equation}
\label{dyn-YQ2}
\Y(t)\mathcal{Q}=\mathcal{Q}\Y_0+\int_0^t\C(\Y(s)\mathcal{Q})\,ds+\mathcal{Q}\Le(t)\,.
\end{equation}
Introduce now the notation 
\begin{equation}
\Le_{\mathcal{Q}}(t):=\mathcal{Q}\Le(t)\,,
\end{equation} 
which is an $\mathcal{H}$-valued process. It is in fact a L\'evy process with values in $\mathcal{H}$. Indeed,
its conditional characteristic function is (here $\mathcal{T}\in\mathcal{H}$ and $t\geq s$)
\begin{align*}
\E\left[\e^{\mathrm{i}\langle\Le_{\mathcal{Q}}(t)-\Le_{\mathcal{Q}}(s),\mathcal{T}\rangle_{\mathcal{H}}}\,|\,\mathcal{F}_s\right]
&=\E\left[\e^{\mathrm{i}\langle\mathcal{Q}(\Le(t)-\Le(s)),\mathcal{T}\rangle_{\mathcal{H}}}\,|\,\mathcal{F}_s\right] \\
&=\E\left[\e^{\mathrm{i}\langle\Le(t)-\Le(s),\mathcal{Q}\mathcal{T}\rangle_{\mathcal{H}}}\,|\,\mathcal{F}_s\right] \\
&=\E\left[\e^{\mathrm{i}\langle\Le(t)-\Le(s),\mathcal{Q}\mathcal{T}\rangle_{\mathcal{H}}}\right] \\
&=\exp\left((t-s)\Psi_{\Le}(\mathcal{Q}\mathcal{T})\right)
\end{align*}
by the independent increment property and the definition of the cumulant of $\Le$. Hence, $\Le_{\mathcal{Q}}(t)-\Le_{\mathcal{Q}}(s)$ is independent of $\mathcal{F}_s$ with a
stationary distribution, which implies that $\Le_{\mathcal{Q}}$ is a L\'evy process. Its covariance operator is given by $\mathbb{Q}_{\Le_{\mathcal{Q}}}=\mathcal{Q}\mathbb{Q}_{\Le}\mathcal{Q}$, which is easily seen from
\begin{align*}
\E\left[\langle\Le_{\mathcal{Q}}(t),\mathcal{T}\rangle_{\mathcal{H}}\langle\Le_{\mathcal{Q}}(t),\mathcal{S}\rangle_{\mathcal{H}}\right]
&=\E\left[\langle\Le(t),\mathcal{Q}\mathcal{T}\rangle_{\mathcal{H}}\langle\Le(t),\mathcal{Q}\mathcal{S}\rangle_{\mathcal{H}}\right] \\
&=\langle\mathbb{Q}_{\Le}\mathcal{Q}\mathcal{T},\mathcal{Q}\mathcal{S}\rangle_{\mathcal{H}} \\
&=\langle\mathcal{Q}\mathbb{Q}_{\Le}\mathcal{Q}\mathcal{T},\mathcal{S}\rangle_{\mathcal{H}} \,,
\end{align*}
with $\mathcal{T},\mathcal{S}\in\mathcal{H}$. Therefore, we have a mild solution of the equation for $\Y_{\mathcal{Q}}:=\mathcal{Q}\Y$ in \eqref{dyn-YQ1} given as
\begin{equation}
\label{dyn-YQ-mild}
\Y_{\mathcal{Q}}(t)=\Sg(t)\mathcal{Q}\Y_0+\int_0^t\Sg(t-s)\,d\Le_{\mathcal{Q}}(s)\,.
\end{equation}
Moreover, we see that $\Y(t)\mathcal{Q}$ in \eqref{dyn-YQ2} solves the same equation, and thus $\mathcal{Q}\Y(t)=\Y(t)\mathcal{Q}$
by uniqueness of solutions, and the claimed commutativity follows.
We remark that if $\mathcal{Q}$ commutes with $\mathcal{C}$, then the assumed property of $\C=\C_i$ holds for $i=1,2$. Also,
if $\Le$ is the simple choice as in Ex.~\ref{trivial-ex-levy}, it commutes with $\mathcal{Q}$ whenever $\mathcal{U}$ commutes with
$\mathcal{Q}$. 

Let us investigate the "adjusted returns" implied by the model. To this end, fix $\Delta t>0$, and define the "adjusted return" by 
$$
R(t,\Delta t)=X(t+\Delta t)-\mathcal{S}(\Delta t)X(t)\,.
$$
From \eqref{X-sol}, we find after using the semigroup property of $\mathcal{S}$,
\begin{align*}
R(t,\Delta t)
&=\int_t^{t+\Delta t}\mathcal{S}(t+\Delta t-s)\Y^{1/2}(s)\,dB(s)\,.
\end{align*}
We have:
\begin{lemma}\label{Lemmafiltering}
Let $\mathcal{F}^{\Y}$ be the $\sigma$-algebra generated by the paths of $\Y$. Then $R(t,\Delta t)|\mathcal{F}^{\Y}$
is a mean zero $H$-valued Gaussian random variable, with covariance operator
$$
\mathcal{Q}_{R(t,\Delta t)|\Y}:=\int_t^{t+\Delta t}\mathcal{S}(t+\Delta t-s)\Y^{1/2}(s)\mathcal{Q}\Y^{1/2}(s)\mathcal{S}^*(t+\Delta t-s)\,ds\,.
$$
\end{lemma}
\begin{proof}
By inspection of the proof of Prop.~\ref{prop:cumulant-X}, we find for $f\in H$
\begin{align*}
&\E\left[\exp(\mathrm{i}(R(t,\Delta t),f)_H)\,|\,\mathcal{F}^{\Y}\right]=\exp\left(-\frac12\int_t^{t+\Delta t}|\mathcal{Q}^{1/2}\Y^{1/2}\mathcal{S}^*(t+\Delta t-s)f|^2_H\,ds\right)\,.
\end{align*}
This is the characteristic function of a Gaussian mean-zero real valued random variable. Hence, $R(t,\Delta t)|\mathcal{F}^{\Y}$ is
Gaussian in $H$ with mean equal to zero. The conditional covariance operator follows by a direct computation.
\end{proof}
The stochastic volatility model yields a Gaussian variance-mixture model for the adjusted returns (see Barndorff-Nielsen and Shephard~\cite{BNS}
for mean-variance mixture models and stochastic volatility in finance). Remark that if $\mathcal{Q}$ and $\Y$ commute, the conditional
covariance operator becomes
$$
\mathcal{Q}_{R(t,\Delta t)|\Y}:=\int_t^{t+\Delta t}\mathcal{S}(t+\Delta t-s)\Y_{\mathcal{Q}}(s)\mathcal{S}^*(t+\Delta t-s)\,ds\,,
$$
with the definition of $\Y_{\mathcal{Q}}$ given above.
\begin{remark} \label{Remarkfiltering} In Lemma \ref{Lemmafiltering} a simplified filtering problem for the adjusted returns of  the model  is solved. Here the observable is the volatility $\Y^{1/2}$. Compared to more general filtering models, as e.g. those described in Xiong~\cite{Xiong}, our filtering model is simple, as the observable does not depend on the signal  $X$, for which the adjusted returns are computed.
\end{remark}

\section{Application to forward price modelling}

Let $H=H_w$, the Filipovic space of all absolutely continuous functions $f:\R_+\rightarrow\R$ such that
\begin{equation}
|f|^2_w:=f(0)^2+\int_0^{\infty}w(x)|f'(x)|^2\,dx<\infty\,,
\end{equation}
where $w:\R_+\rightarrow \R_+$ is an increasing function with $w(0)=1$. We assume that $\int_0^{\infty}w^{-1}(x)\,dx<\infty$, and
denote the (naturally defined) inner product $(\cdot,\cdot)_w$.  
It turns out that $H_w$ is a separable Hilbert space equipped with the norm $|\cdot|_w$. Moreover, the evaluation functional
$\delta_x(f)=f(x)$ is continuous on $H_w$. As a linear functional, we can express $\delta_x$ by $(\cdot,h_x)_w$, with
\begin{equation}
\label{eq:def-hx}
h_x(y)=1+\int_0^{x\wedge y}w^{-1}(z)\,dz, y\in\R_+\,.
\end{equation}
See Filipovic~\cite{filipovic} for the introduction of this space and its properties (see also Benth and Kr\"uhner~\cite{BK-HJM} for a further
analysis of this space). 

Consider $X$ defined in \eqref{X-dyn} for $H=H_w$ and $A=\partial/\partial x$, the derivative operator. Then $X$ can be considered
as the dynamics of  the forward curve, that is, $f(t,x):=\delta_x(X(t))=X(t)(x)$, where $f(t, x):=F(t, t+x)$, and $t\mapsto F(t,T), t\leq T$ is the arbitrage-free forward price dynamics of a contract delivering an asset (commodity or stock) at time $T$ (see Benth and Kr\"uhner~\cite{BK-HJM}). 
We note that the semigroup of $A$ will be the right shift operator
$\mathcal{S}(t)f=f(\cdot+t)$, and that
$$
\delta_x\mathcal{S}(t)g=g(t+x)=\delta_{x+t}g\,.
$$ 
for any $g\in H_w$. We find from the mild solution of $X$ in \eqref{X-sol} that
\begin{equation}
\label{eq:f-dyn-filipovic}
f(t,x)=f_0(t+x)+\delta_x\int_0^t\mathcal{S}(t-s)\Y^{1/2}(s)\,dB(s)
\end{equation}
where $f_0(t+x)=\delta_x\mathcal{S}(t) X_0$. Note that by Lemma~3.2 in Benth and Kr\"uhner~\cite{BK-HJM}, it holds
$$
\lim_{t\rightarrow\infty}(X_0,\mathcal{S}^*(t)h_x)_H=\lim_{t\rightarrow\infty}\delta_x\mathcal{S}(t)X_0=\lim_{t\rightarrow\infty}f_0(t+x)
=f_0(\infty)\,.
$$
Here, $f_0(\infty)$ denotes the limit of $f_0(y)$ as $y\rightarrow\infty$, which exists. Hence, from the mild solution in 
\eqref{eq:f-dyn-filipovic}, the mean of $f(t,x)$ for given $x\in\R_+$ has a limit $f_0(\infty)$ as time tends to infinity.


Now we investigate the stochastic integral in \eqref{eq:f-dyn-filipovic} in more detail.
By Thm.~2.1. in Benth and Kr\"uhner~\cite{BK-HJM}, there exists a real-valued Brownian motion $b_x$ such that
\begin{equation}
\delta_x\int_0^t\mathcal{S}(t-s)\Y^{1/2}(s)\,dB(s)=\int_0^t\sigma(t,s,x)\,db_x(s)\,,
\end{equation}
where
\begin{align*}
\sigma^2(t,s,x)&=(\delta_x\mathcal{S}(t-s)\Y^{1/2}(s)\mathcal{Q}(\mathcal{S}(t-s)\Y^{1/2}(s))^*\delta_x^*)(1) \\
&=\delta_x\mathcal{S}(t-s)(\Y^{1/2}(s)\mathcal{Q}\Y^{1/2}(s))(\delta_x\mathcal{S}(t-s))^*(1) \\
&=\delta_{x+t-s}(\Y^{1/2}\mathcal{Q}\Y^{1/2}(s))\delta_{x+t-s}^*(1)\,.
\end{align*}
We remark that the Brownian motion $ b_x$ depends on $x$, since the representation in Thm.~2.1. in Benth and Kr\"uhner~\cite{BK-HJM}
is for a given linear functional, which in this case $\delta_x$. We know that 
$\delta_x^*(1)=h_x(\cdot)$  (see e.g. Benth and Kr\"uhner~\cite{BK-HJM}), and therefore
\begin{equation}
\sigma^2(t,s,x)=(\Y^{1/2}(s)\mathcal{Q}\Y^{1/2}(s)(h_{x+t-s}(\cdot)))(x+t-s)\,.
\end{equation}
Hence, we map the function $h_{x+t-s}$ by the operator $\Y^{1/2}(s)\mathcal{Q}\Y^{1/2}(s)$, and evaluate the resulting function
in $H_w$ at $x+t-s$. As $\Y^{1/2}$ is stochastic, we get a stochastic volatility $\sigma(t,s,x)$, which is depending on current time $t$,
previous times $s$ and the "spatial" variable $x$. 
In particular, the spot price dynamics $S(t):=f(t,0)$ becomes
$$
S(t)=f_0(t)+\int_0^t\sigma(t,s,0)\,db_0(s)\,.
$$
I.e., the spot price dynamics follows a Volterra process where the integrand $\sigma(t,s,0)$ is stochastic. We refer to 
Barndorff-Nielsen, Benth and Veraart~\cite{BNBV-spot} for an application to Volterra processes (and more specifically, Brownian and
L\'evy semistationary processes) to model spot prices in energy markets.

Let us carry our discussion further, and {\it suppose} that $\mathcal{Q}$ commutes with $\Y_0$ and $\Le(t)$ for $t\geq 0$, as well 
as that we have $\C(\mathcal{T})\mathcal{Q}=\C(\mathcal{T}\mathcal{Q})$ and $\mathcal{Q}\C(\mathcal{T})=\C(\mathcal{Q}\mathcal{T})$
for any $\mathcal{T}\in\mathcal{H}$. 
Then we recall from the previous Section that $\Y(s)$ will commute with $\mathcal{Q}$ for every $s\geq 0$. The process $\Y^{1/2}(s)$ will also commute with $\mathcal{Q}$, and 
$$
\sigma^2(t,s,x)=(\Y(s)\mathcal{Q}h_{x+t-s})(x+t-s)\,.
$$
Recalling the definition of $\Y_{\mathcal{Q}}$ in \eqref{dyn-YQ-mild}, we find
\begin{align*}
\sigma^2(t,s,x)&=\delta_{x+t-s}(\Y_{\mathcal{Q}}(s)(h_{x+t-s})) \\
&=(\Y_{\mathcal{Q}}(s)h_{x+t-s},h_{x+t-s})_w \\
&=\langle\Y_{\mathcal{Q}}(s),h_{x+t-s}\otimes h_{x+t-s}\rangle_{\mathcal{H}}
\end{align*}
since $\delta_z(f)=(f,h_z)_w$ for any $f\in H_w$. Similar as in Prop.~\ref{prop:cumulant-Y}, we can calculate the cumulant of the process
$\Y_{\mathcal{Q}}$ for any $\mathcal{T}\in\mathcal{H}$, and in particular we can calculate the cumulant of the process
$s\mapsto \langle\Y_{\mathcal{Q}}(s),h_{x+t-s}\otimes h_{x+t-s}\rangle_{\mathcal{H}}$ for $s\leq t$ by choosing 
$\mathcal{T}=h_{x+t-s}\otimes h_{x+t-s}$. A simple calculation using the definition of $h_x$ in
\eqref{eq:def-hx} shows that 
$$
(h_{x+t-s}\otimes h_{x+t-s})(f)=\left(f(0)+\int_0^{x+t-s}f'(y)\,dy\right)h_{x+t-s}=\mathcal{I}_{x+t-s}(f)h_x\,,
$$ 
where $\mathcal{I}_x\in H_w^*$ is defined as $\mathcal{I}_x(f)=\delta_0(f)+\int_0^xf'(y)\,dy$ for any $f\in H_w$. 

In the above considerations we obtain a "marginal" dynamics, in the sense of a dynamics for a forward contract with fixed time to maturity
$x$. We now represent the forward price dynamics as a space-time random field to emphasize also its spatial
dynamics (i.e., its dynamics in time-to-maturity $x$). First, from Prop.~3.6 in 
Benth and Kr\"uhner~\cite{BK2} we find for any $f\in H_w$,
\begin{align*}
(\mathcal{S}(t-s)\Y^{1/2}(s)f)(x)&=(\mathcal{S}(t-s)\Y^{1/2}(s)f,h_x)_w \\
&=(f,(\mathcal{S}(t-s)\Y^{1/2}(s))^*(h_x))_w \\
&=(\mathcal{S}(t-s)\Y^{1/2}(s))^*(h_x)(0)f(0) \\
&\qquad\qquad+\int_0^{\infty}w(y)(\mathcal{S}(t-s)\Y^{1/2}(s))^*(h_x)'(y)f'(y)\,dy \\
&=(\Y^{1/2}(s)\mathcal{S}^*(t-s)h_x)(0)f(0) \\
&\qquad\qquad+\int_0^{\infty}w(y)(\Y^{1/2}(s)\mathcal{S}^*(t-s)h_x)'(y)f'(y)\,dy \,.
\end{align*}
Again from Prop.~3.6 in Benth and Kr\"uhner~\cite{BK2},
\begin{align*}
\mathcal{S}^*(t)h_x(\cdot)&=h_x(0)(\mathcal{S}(t)h_{\cdot})(0)+\int_0^{\infty}w(y)(\mathcal{S}(t)h_{\cdot})'(y)h_x'(y)\,dy\,.
\end{align*}
But $\mathcal{S}(t)h_x(y)=h_x(y+t)$ and $h_x'(y)=w^{-1}(y)\mathbf{1}(y<x)$. Hence,
\begin{align*}
\mathcal{S}^*(t)h_x(\cdot)&=h_{\cdot}(t)+\int_0^xw^{-1}(y+t)\mathbf{1}(y+t<\cdot)\,dy \\
&=h_t(\cdot)+\int_t^{x+t}w^{-1}(y)\mathbf{1}(y<\cdot)\,dy \\
&=h_{t+x}(\cdot)\,.
\end{align*}
If we use the notation that $B(ds,dy):=\partial_xB(ds,y)\,dy$, we find
\begin{align*}
\delta_x\int_0^t\mathcal{S}(t-s)\Y^{1/2}(s)\,dB(s)&=\int_0^t(\Y^{1/2}(s)h_{x+t-s})(0)\,dB(s,0) \\
&\qquad+\int_0^t\int_0^{\infty}
w(y)(\Y^{1/2}(s)h_{x+t-s})'(y)\,B(ds,dy)\,,
\end{align*}
or, a representation of $f(t,x):=\delta_x(X(t))$ as a spatio-temporal random field
\begin{align*}
f(t,x)&=f_0(t+x)+\int_0^t(\Y^{1/2}(s)h_{x+t-s})(0)\,dB(s,0)+\int_0^t\int_0^{\infty}
w(y)(\Y^{1/2}(s)h_{x+t-s})'(y)\,B(ds,dy)\,.
\end{align*}
Note that $B(t,0)=\delta_0B(t)=(B(t),h_0)_w=(B(t),1)_w$ is a real-valued Brownian motion with variance $|\mathcal{Q}^{1/2}1|^2_w$.
Hence, $b_0(t):=B(t,0)/|\mathcal{Q}^{1/2}1|_w$ is a real-valued standard Brownian motion and we can view the first integral as 
an Ito integral of a volatility process given by $s\mapsto (\Y^{1/2}(s)h_{x+t-s})(0)$ for $s\leq t$ where $x$ is a parameter. 
It becomes a real-valued Volterra process with parameter $x$. For the second
integral, we integrate with respect to a spatio-temporal random field $(s,y) \mapsto B(s,y)$ over $[0,t]\times\R_+$, thus
becoming a stochastic
Volterra random field. This part is analogous to an ambit field, a class of spatio-temporal random fields defined in
Barndorff-Nielsen and Schmiegel~\cite{BN-Schm}. In a special case, the ambit fields take the form
$$
A(t,x)=\int_0^t\int_{0}^{\infty}g(t,s,x,y)\eta(s,y) B(ds,dy)
$$
for a stochastic random field $\eta$ and a deterministic kernel function $g$. Under appropriate integrability conditions, the
ambit field $A(t,x)$ is well-defined (see e.g. Barndorff-Nielsen, Benth and Veraart~\cite{BNBV-first}). We observe that 
we can identify $w(y)$ with the kernel function $g$, giving a very simple kernel. On the other hand, the volatility
field $\eta$ is more complex in $f$, as it is also $x$-dependent and not only $s$ and $y$ dependent. Our stochastic
volatility model serves as a motivation for an extension of the ambit field models. We refer to  
Barndorff-Nielsen, Benth and Veraart~\cite{BNBV-forward} and \cite{BNBV-cross} for
an application of ambit fields to energy forward price modelling.

We finally remark that in many commodity markets one observes an increasing volatility with decreasing time to delivery, known as the 
Samuelson effect (see Samuelson~\cite{samuelson}). To include this in our dynamics of $X$, we can add an operator $\Psi(t)\in\mathcal{H}$,
possibly time-dependent, such that
$$
dX(t)=\mathcal{A}X(t)\,dt+\Psi(t)\Y^{1/2}(t)\,dB(t)\,.
$$ 
Much of the analysis above can, under natural integrability conditions on $\Psi$, be carried through for this model. 

\appendix

\section{A result on symmetric Hilbert-Schmidt operators}
In this section we provide the arguments for a claim in the proof of Proposition \ref{prop:zeroMGpart} for the convenience of the reader. 

\begin{lemma}\label{lem:app1}
  Let $H$ be a separable Hilbert space and let $(e_k)_{k\in\N}$ an orthonormal basis of $H$. Then any symmetric Hilbert-Schmidt operator $\mathcal{V}$ on $H$ can be written as
  \begin{align*}
		\mathcal{V} = \sum_{k,l\in\N} \gamma_{k,l} e_k\otimes e_l,
	\end{align*}
		with $\sum_{k,l} \gamma_{k,l}^2<\infty$ and $\gamma_{k,l} = \gamma_{l,k}$. 
\end{lemma}
\begin{proof}
Recall that the space of Hilbert-Schmidt operators on $H$, denoted by $\mathcal{H}$, is isometrically isomorph to $H^*\otimes H$, which we identify as $H\otimes H$. Therefore, any Hilbert-Schmidt operator $\mathcal{V}$ can be written as
\[ \mathcal{V} = \sum_{k,l\in\N} \gamma_{k,l} e_k\otimes e_l, \]
for a sequence of constants $(\gamma_{k,l})_{k,l\in\N}$. 

As for the square-summability of the constants, we have necessarily
\begin{align*}
	\|\mathcal{V}\|_{\mathcal{H}}^2 
	= \sum_{i\in\N} \bigg|\sum_{k,l\in\N} \gamma_{k,l}(e_k\otimes e_l)e_i \bigg|_H^2 
	& = \sum_{i\in\N} \bigg|\sum_{l\in\N} \gamma_{i,l}e_l\bigg|_H^2 \\
	& = \sum_{i\in\N} \bigg( \sum_{l\in\N} \gamma_{i,l}e_l, \sum_{m\in\N} \gamma_{i,m}e_m\bigg)_H  \\
	& = \sum_{i,m\in\N} \gamma_{i,m}^2 (e_m,e_m)_H \\
	& = \sum_{i,m\in\N} \gamma_{i,m}^2\,.
\end{align*}
So, in order for the norm to be finite, the double sequence $(\gamma_{k,l})_{k,l\in\N}$ has to be square-summable.

As for the property of being symmetric, we need to have for all $f,g\in\dom(\mathcal{V})$ that
\begin{equation}\label{prf:app1}
  \big(\mathcal{V}f,g\big)_H = \big(f,\mathcal{V}g\big)_H.
\end{equation}
Since Hilbert-Schmidt operators are bounded (even compact), one has $\dom(\mathcal{V})=H$. As a side-remark, this furthermore implies that $\dom(\mathcal{V}^*)\supseteq \dom(\mathcal{V}) = H$, which in turn implies that symmetric Hilbert-Schmidt operators are already self-adjoint. The terms in \eqref{prf:app1} can be evaluated to be
\[ \big(\mathcal{V}f,g\big)_H = \sum_{k,l} \gamma_{k,l}f_kg_l \quad\text{and}\quad \big(f,\mathcal{V}g\big)_H = \sum_{k,l} \gamma_{k,l}f_lg_k, \]
where $f_k =(f,e_k)_H$, and similarly for $g_l$. These terms can only be equal for all $f,g\in \dom(\mathcal{V}) = H$ if either $k=l$ or if $\gamma_{k,l}=\gamma_{l,k}$, which implies the assertion. 
\end{proof}

\end{document}